\documentclass[12pt,leqno]{article}
\usepackage{graphicx, fullpage}
\usepackage{amsmath,amssymb,amsthm,amscd, bm}
\usepackage{fancyhdr}
\usepackage[mathscr]{eucal}
\usepackage{amsfonts}
\begin{document}
\parskip=6pt

\theoremstyle{plain}

\newtheorem {thm}{Theorem}[section]
\newtheorem {lem}[thm]{Lemma}
\newtheorem {cor}[thm]{Corollary}
\newtheorem {defn}[thm]{Definition}
\newtheorem {prop}[thm]{Proposition}
\newtheorem{ex}[thm]{Example}
\numberwithin{equation}{section}

\newcommand{\cF}{{\cal F}}
\newcommand{\cA}{{\cal A}}
\newcommand{\cC}{{\cal C}}
\newcommand{\cH}{{\cal H}}
\newcommand{\cU}{{\cal U}}
\newcommand{\cK}{{\cal K}}
\newcommand{\cM}{{\cal M}}
\newcommand{\cO}{{\cal O}}
\newcommand{\cN}{{\cal N}}
\newcommand{\cP}{{\cal P}}
\newcommand{\cR}{{\cal R}}
\newcommand{\bC}{\mathbb C}
\newcommand{\bP}{\mathbb P}
\newcommand{\bN}{\mathbb N}
\newcommand{\bA}{\mathbb A}
\newcommand{\bR}{\mathbb R}
\newcommand{\fg}{\frak g}
\newcommand{\fy}{\frak y}
\newcommand{\fh}{\frak h}
\newcommand{\fV}{\frak V}
\newcommand{\var}{\varepsilon}
\renewcommand\qed{ }
\newcommand{\sgrad}{\text{sgrad }}
\newcommand{\grad}{\text{grad }}
\newcommand{\id}{\text{id}}
\newcommand{\Id}{\text{Id}}
\newcommand{\ad}{\text{ad}}
\newcommand{\const}{\text{const }}
\newcommand{\Ker}{\text{Ker }}
\newcommand{\Hom}{\text{Hom}\,}
\newcommand{\opartial}{\overline\partial}
\newcommand{\Ree}{\text{Re }}
\newcommand{\an}{\text{an}}
\newcommand{\tr}{\text{tr}}

\begin{titlepage}
\title{\bf Riemannian geometry in infinite dimensional spaces\thanks{Research partially supported by NSF grant DMS--1464150\newline 2010 Mathematics subject classification 58B20}}
\author{L\'aszl\'o Lempert\\ Department of  Mathematics\\
Purdue University\\West Lafayette, IN
47907-2067, USA}
\end{titlepage}
\date{}
\maketitle
\abstract

We lay foundations of the subject in the title, on which we build in another paper devoted to isometries in spaces of K\"ahler metrics.

\section{Introduction}

Our modest goal with this paper is to collect foundational material in differential and Riemannian geometry in infinite dimensional spaces, beyond the framework of Banach and Hilbert manifolds as treated by Lang in \cite{La}.
It is unlikely that any of the results or proofs we collect here will surprise the reader; yet this material is necessary to buttress our work \cite{Le} on isometries in spaces of K\"ahler potentials.

We will develop the notions of Riemannian metrics, connections---Levi--Civita or otherwise---, curvature, and geodesics in open subsets of locally convex (topological vector) spaces, rather than in manifolds modeled on such spaces.
This allows for various simple definitions: for example, to define the differential of a form $\alpha$ we need not bring in Cartan's formula, general vector fields and their Lie brackets, and then verify that the formula indeed produces a form $d\alpha$.
If the definitions become simpler, it becomes more complicated, though, to elucidate how the notions defined transform under smooth maps.
So, in addition to giving the definitions, what we mainly do in this paper is to justify transformation rules familiar from finite dimensional geometry.  Overall, when the transformations are $C^2$ the proofs are quite straightforward, but they are less so when the transformations are only $C^1$.

The papers \cite{H,Mi} (and undoubtedly many others) have some overlap with this one, although their perspectives and goals are different, and do not delve into Riemannian matters.

\section{Calculus}

It will not hurt to start by going over the main notions of differential calculus in Fr\'echet and more general spaces.
Given a real vector space $V$, a family $\cP$ of seminorms endows it with the structure of a locally convex (topological vector) space: a neighborhood basis of $0\in V$ consists of intersections of finitely many balls $\{v\in V\colon p(v)<\var\}$, $p\in\cP$.
All our locally convex spaces will be assumed Hausdorff and sequentially complete.
For $(V,\cP)$ this latter means that if $v_n\in V$, $n\in\bN$, and $v_n-v_m\to 0$ as $n,m\to\infty$, then the $v_n$ converge.

If $W$ is another locally convex space and $\Omega\subset V$ is open, a map $f\colon\Omega\to W$ is $C^1$ if its directional derivatives
$$
df(v,\xi)=\lim_{t\to 0}\ {f(v+t\xi)-f(v)\over t},\qquad v\in\Omega,\quad\xi\in V
$$
exist and $df\colon \Omega\times V\to W$ is continuous.
It then follows that $df(v,\xi)$ is linear in $\xi$.
(For Banach spaces this is a slightly weaker requirement than the one in the standard definition, to wit, $v\mapsto df(v,\cdot)\in \Hom (V,W)$ should be continuous in the norm topology on $\Hom(V,W)$.)
If $df$ is not only continuous but $C^1$, we say that $f$ is $C^2$, and so on. If $S$ is an arbitrary subset of $V$, by $C^k(S,W)$ we mean the space of functions $f:S\to W$ that extend to a $C^k$ function in some open neighborhood $\Omega\supset S$. 
We write $C(S,W)=C^0(S,W)$ for continuous maps $S\to W$ and $C^\infty(S,W)$ for $\bigcap^\infty_{k=1} C^k (S,W)$. Often it is convenient to write $(\xi f)(v)$ for $df(v,\xi)$ and $\xi\eta f$ for $\xi(\eta f)$, etc.

\begin{ex}If $\phi\colon V^{\oplus r}\to W$ is continuous and multilinear, the function $f\colon V\ni v\mapsto \phi (v,v,\ldots,v)\in W$ is $C^\infty$.
\end{ex}

Indeed, the directional derivatives
$$
df(v,\xi)=\phi(\xi,v,\ldots,v)+\phi(v,\xi,v,\ldots,v)+\ldots+\phi (v,v,\ldots,v,\xi)
$$
exist, and $df\colon V\oplus V\to W$ is continuous. It is also the restriction of a multilinear form to the diagonal, and
the smoothness of $f$ follows by induction.

Functions $f$ of the form above are called homogeneous polynomials, and sums of finitely many such functions are called (continuous) polynomials.

For the next example let $R=\prod_1^n (a_i,b_i)\subset\bR^n$ be a (bounded) rectangular box. It is easy to check that $u\colon\bar R\to V$ is
in $C^k(\bar R,V)$ for some $k=0,1,\ldots$ if and only if it is continuous and all partial derivatives $\partial^\alpha (u|R)$ of order $|\alpha|\le k$ extend continuously to $\bar R$. Any $p\in\cP$ and multiindex $\alpha$ of length $\le k$ induces a seminorm
$$
p_\alpha(u)=\sup_R p(\partial^\alpha u) 
$$ 
on $C^k(\bar R,V)$, and these seminorms  together define a locally convex topology on $C^k(\bar R,V)$, that is sequentially complete. If $\Omega\subset V$ is open, then 
$$
C^k(\bar R,\Omega)=\{u\in C^k(\bar R,V)\colon u(\bar R)\subset\Omega\}
$$
is open in $C^k(\bar R,V)$.

\begin{ex}If $l=0,1,\ldots$ and $f\in C^{k+l}(\Omega,W)$, then the function
$$
F\colon C^k(\bar R,\Omega)\ni u\mapsto f\circ u\in C^k(\bar R,W)
$$
is $C^l$.
\end{ex}

This is first proved for $l=0$ by induction on $k$. If $l=1$, one computes that with $u\in C^k(\bar R,\Omega)$, $\zeta\in C^k(\bar R,V)$
\begin{equation}
dF(u,\zeta)=\lim_{t\to 0}\frac{f(u+t\zeta)-f(u)}t=(df)\circ(u,\zeta)\in C^k(\bar R, W),
\end{equation}
so $dF$ is continuous by the $l=0$ case. Finally, for general $l$ the result follows from (2.1) by induction.

We will not need the general notion of manifolds modeled on $V$, but we will introduce the tangent bundle $T\Omega$ as one does in a manifold: tangent vectors are equivalence classes of $C^1$ maps $I\to\Omega$, with $I$ some neighborhood of $0\in\bR$.
Two $C^1$ maps $f\colon I\to\Omega$, $g\colon J\to\Omega$ are equivalent if $f(0)=g(0)$ and $df(0,\cdot)=dg(0,\cdot)$.
The map
$$
T\Omega\ni [f]\mapsto (f(0), df(0,1))\in\Omega\times V
$$
is a bijection.
We will typically identify $T\Omega$ with the open subset $\Omega\times V\subset V\oplus V$, and the tangent spaces $T_v\Omega=\{[f]\colon f(0)=v\}$ with $V$ itself.
This allows us to talk about smoothness of maps between tangent bundles.
A map $F\in C^k (\Omega,W)$ induces a $C^{k-1}$ map between tangent bundles, denoted $F_*$,
$$
F_*\colon T\Omega\ni [f]\mapsto [F\circ f]\in TW,
$$
and we write $F_*|_v$ for its restriction $T_v\Omega\to T_{F(v)} W$.
In the identification above, $F_*(v,\xi)=(F(v), dF(v,\xi))$.

Because of the identifications $T_v\Omega\approx V$, $T_{F(v)} W\approx W$, often we will use $F_*|_v$ to denote the corresponding linear map $V\to W$, i.e., $F_*|_v=df(v,\cdot)$.

We will also need calculus of functions that take values in $\Hom (W,Z)$, the vector space of continuous linear maps between locally convex spaces $W,Z$.
There are several natural topologies on this space but  none of them is of any use for differential geometry unless $W$ is Banach, cf.~\cite{Ma}.
Accordingly, we forgo introducing a topology on $\Hom (W,Z)$, and define $f\colon\Omega\to\Hom (W,Z)$ to be $C^k$ if the map
$$
\Omega\times W\ni (v,w)\mapsto f(v) w\in Z
$$
is $C^k$.
We express this by writing $f\in C^k(\Omega,\Hom(W,Z))$.
If $f\in C^1(\Omega,\Hom (W,Z))$, its differential $df\colon\Omega\times V\to\Hom (W,Z)$ is defined by
$$
df(v,\xi) w=\lim_{t\to 0}\ {f(v+t\xi) w-f(v)w\over t}.
$$

This language restores, to a certain extent, the analogy with Fr\'echet's notion of differentiability of maps between Banach spaces: in our general setting, $f\colon\Omega\to W$ is $C^k$ if the map $v\mapsto df(v,\cdot)\in\Hom (V,W)$ is in $C^{k-1}(\Omega,\Hom (V,W))$.

One can check that if $f$ is $C^k$ on $\Omega$, and $g\colon\Omega'\to\Omega$ is a $C^k$ map of an open $\Omega'\subset V'$, $V'$ locally convex, then $f\circ g$ is $C^k$, and the chain rule holds. Once the chain rule is known, it follows that if $f\colon\Omega\to\Hom(W,Z)$ and $g\colon\Omega\to\Hom(Z,X)$ are $C^k$, then
$$
(gf)(v)=g(v)f(v)\in\Hom(W,X),\qquad v\in\Omega,
$$ 
defines a $C^k$ function $\Omega\to\Hom(W,X)$, whose differential is
$$
d(gf)(v,\xi)=dg(v,\xi)f(v)+g(v)df(v,\xi).
$$

We will also need the notion of integral. We will only integrate piecewise continuous $V$ valued functions on intervals, and then the integral can be defined as the limit of Riemann 
sums.---To finish this section here are two criteria for maps to be $C^k$. We fix $k\in\bN$.
\begin{lem}Let  $I\subset\bR$  be an open interval. A continuous function $h\colon I\to W$ is $C^k$ if and only
\begin{equation}
g(\tau)=\lim_{\sigma\to 0}\frac{\sum_{j=0}^k (-1)^{k-j}\binom kj h(\tau+j\sigma)}{\sigma^k}
\end{equation}
defines a continuous function $g\colon I\to W$. In this case $d^kh/d\tau^k=g$.
\end{lem}
\begin{proof}The only if direction follows from the identity
$$
\sum_{j=0}^k (-1)^{k-j}\binom kj h(\tau+j\sigma)=\int_0^{\sigma}\ldots\int_0^{\sigma}\frac{d^kh}{d\tau^k}(\tau+s_1+\ldots+s_k)ds_1\ldots ds_k.
$$
As to the opposite direction,  integrating $g$ $k$ times gives an $H\in C^k( I, Z)$ such that $d^kH/d\tau^k=g$. 
Upon replacing $h$ by $h-H$ we can reduce ourselves to the case $g=0$, when we need to show that $h$ is a polynomial of degree $<k$.
 We convolve $h$ with mollifiers $\chi_\varepsilon$. Then $h_\var=h*\chi_\var$  satisfies (2.2), still with $g=0$.
 But now $h_\var$ is smooth, so this implies $d^kh_\var/d\tau^k=0$; therefore $h_\var$ is a polynomial of degree $<k$. Choosing $\chi_\var d\tau$ to converge weak$^*$ to the Dirac measure, we then obtain that $h=\lim h_\var$ is also a polynomial of degree $<k$.
\end{proof}

If $u$ is a function on an open subset of $\bR^\nu$, we write $\partial_j u$ for its derivative with respect to the $j$'th variable, and if $\rho\in\bR^\nu$, $\partial_\rho$ for $\sum_j\rho_j\partial_j$. 
Given $k\in\bN$, fix a finite $T\subset\bR^\nu$ so that polynomials on $\bR^\nu$ of degree $\le k$ can be recovered from their restrictions to $T$, and for multiindices $(j_1,\ldots,j_\nu)$ choose functions $a_{j_1\ldots j_\nu}\colon T\to\bR$ so that for such polynomials $p$
$$
p(t)=\sum_{j_i\ge 0}\sum_{\rho\in T}a_{j_1\ldots j_\nu}(\rho)p(\rho)t_1^{j_1}\ldots t_\nu^{j_\nu},\qquad t=(t_i)\in\bR^\nu.
$$ 
\begin{lem}Suppose $D\subset\bR^\nu$ is open, $Z$ is a locally convex topological vector space, and $u:D\to Z$ is continuous. 
If $\partial_\rho^lu(t)$ exists and defines a continuos function of $(t,\rho)\in D\times\bR^\nu$ for all $l\le k$, then $u$ is $C^k$ and when $j_1+\ldots+j_\nu=k$
\begin{equation}
\partial_1^{j_1}\ldots\partial_\nu^{j_\nu} u= \sum_{\rho\in T}a_{j_1\ldots j_\nu}(\rho) \partial_\rho^ku.
\end{equation}
\end{lem}
\begin{proof}Once $u$ is known to be  $C^k$, (2.3) follows if we expand $\partial_t^k$ by the multinomial theorem and collect like terms.
 In general, we convolve $u$ by mollifiers $\chi_\var$; then (2.3) holds with $u$ replaced by $u_\var=\chi_\var *u$. Hence  if we choose $\chi_\var$ to approximate unity,
$$
\partial_1^{j_1}\ldots\partial_\nu^{j_\nu}u_\var= \sum_{\rho\in T}a_{j_1\ldots j_\nu}(\rho) \partial_\rho^k u_\var\\
=\chi_\var* \sum_{\rho\in T}a_{j_1\ldots j_\nu}(\rho) \partial_\rho^ku
\to  \sum_{\rho\in T}a_{j_1\ldots j_\nu}(\rho) \partial_\rho^ku
$$
as $\var\to 0$, locally uniformly on $D$. This last function, call it $v$, is continuous. Arguing by induction, we can assume that $u$ is $C^{k-1}$. If, say, $j_\nu\ge 1$,
$$
\partial_1^{j_1}\ldots\partial_\nu^{j_\nu-1}u\Big|_{(t_1,\ldots,t_{\nu-1},a)}^{(t_1,\ldots,t_{\nu-1},b)}
=\lim_{\var\to 0} \partial_1^{j_1}\ldots\partial_\nu^{j_\nu-1}u_\var\Big|_{(t_1,\ldots,t_{\nu-1},a)}^{(t_1,\ldots,t_{\nu-1},b)}=\int_a^b v(t)dt_\nu.
$$
This shows that $\partial_1^{j_1}\ldots\partial_\nu^{j_\nu-1}u$ has a continuous $\partial_\nu$ derivative, and the lemma follows.  
\end{proof}

\section{Differential forms}

Let $V,W,Z$ be locally convex spaces and $\Omega\subset V$ open, as before.
Given $r=0,1,2,\ldots$, a $W$ (resp.~$\Hom (W,Z)$) valued $r$--form on $\Omega$ is a map
$$
A\colon \Omega\times V^{\oplus r}\to W\qquad (\text{resp. }\Hom (W,Z))
$$
that is alternating $r$--linear on each $\{v\}\times V^{\oplus r}$.
We write $A\in C_r^k (\Omega,W)$ if the $W$ valued map $A$ above is $C^k$, and $A\in C_r^k (\Omega,\Hom (W,Z))$ if the map
$$
\Omega\times V^{\oplus r}\times W\ni (v,\xi_1,\ldots,\xi_r,w)\mapsto A(v,\xi_1,\ldots,\xi_r) w\in Z
$$
is $C^k$.

Forms can be pulled back along $C^1$ maps.
If $V'$ is another locally convex space, $F\colon\Omega\to V'$ is a $C^1$ map, and $B$ is a $W$ or $\Hom (W,Z)$ valued $r$--form in a neighborhood of $F(\Omega)\subset V'$, the pullback $F^*B$ is defined by
$$
(F^*B)(v,\xi_1,\ldots,\xi_r)=B(F(v),F_*|_v\xi_1,\ldots,F_*|_v\xi_r).
$$
If $B$ is $C^k$ and $F$ is $C^{k+1}$, then $F^*B$ is $C^k$.

Given a $C^1$ $r$--form $A$, its exterior derivative is the $(r+1)$--form $dA$,
\begin{equation}
dA(\cdot,\xi_0,\ldots,\xi_r)=\sum_j (-1)^j \xi_j A(\cdot,\xi_0,\ldots\xi_{j-1},\xi_{j+1},\ldots,\xi_r).
\end{equation}
The following lemma, if predictable, is not entirely obvious:

\begin{lem}Pullback and exterior differentiation commute.
That is, if $F\in C^2(\Omega,V')$ and $B$ is a $W$ or $\Hom (W,Z)$ valued $r$--form of class $C^1$ in a neighborhood of $F(\Omega)$, then
\begin{equation}
F^* dB=dF^*B.
\end{equation}
\end{lem}

\begin{proof}First consider a $W$ valued form $B$.
Continuous linear forms $l\colon W\to\bR$ commute with pullback and $d$; since two vectors in $W$ must be equal if each such $l$ takes the same values on them, it suffices to prove when $W=\bR$.

Given $v\in\Omega$, let us say that maps $F,G\colon\Omega\to V'$ agree to order 0 at $v$ if $F(v)=G(v)$, and inductively, that they agree to order $k=1,2,\ldots$ at $v$ if $dF$ and $dG$ agree to order $k-1$ at all points of $T_v\Omega$.
Thus agreement to order $k$ makes sense if $F,G$ are $C^k$.
If now $F$ and $G$ agree to order 2 at $v$, then $F^*dB$, $G^*dB$ on the one hand, and $dF^*B$, $dG^*B$ on the other agree at $v$.
Hence to prove (3.2) at $v$, we are free to replace $F$ by a $G$ that agrees with it to order 2 at $v$.
We take $G$ to be the second order Taylor polynomial of $F$ at $v$, that is obtained by defining
$$
d^2 F(v,\xi,\eta)=\lim_{t\to 0}\ {dF(v+t\eta,\xi)-dF(v,\xi)\over t},
$$
and letting
$$
G(v+\xi)=F(v)+dF (v,\xi)+{1\over 2}\ d^2 F(v,\xi,\xi).
$$
The advantage of working with a polynomial map like $G$ is that it maps a finite dimensional space into a finite dimensional space, and so we can justify
$$
(G^* dB) (v,\xi_0,\ldots,\xi_r)=(d G^* B) (v,\xi_0,\ldots,\xi_r)
$$
by invoking the corresponding finite dimensional result.

The case of $\Hom (W,Z)$ valued forms $B$ can be reduced to what we have already proved by noting a $1-1$ correspondence between $B\in C_r^k(\Omega,\Hom (W,Z))$ and $b\in C_r^k (\Omega\times W,Z)$ that are linear in $w$.
Writing $\pi\colon\Omega\times W\to\Omega$ for the projection, the correspondence is
$$
b(v,w,\eta_1,\ldots,\eta_r)=B(v,\pi_*\eta_1,\ldots,\pi_*\eta_r)(w),\qquad (v,w)\in\Omega\times W,\ \eta_1,\ldots,\eta_r\in V\times W.
$$
Our $F$ induces a map $f=F\times\id_W\colon \Omega\times W\to V'\times W$, and then the $\Hom (W,Z)$ valued part of Lemma 3.1 follows from $f^*db=df^*b$.
\end{proof}

Wedge products can also be defined in infinite dimensional spaces.
All we need is the product of two 1--forms.
Suppose $X$ is yet another locally convex space, and $B,A$ are $\Hom(W,Z)$ resp.~$\Hom (Z,X)$ valued 1--forms on $\Omega$.
Then $A\wedge B$ is a $\Hom (W,X)$ valued 2--form,
$$
(A\wedge B) (v,\xi,\eta)=A(v,\xi) B(v,\eta)-A(v,\eta) B(v,\xi),\qquad v\in\Omega,\ \xi,\eta\in V.
$$

\section{Connection and curvature}

With notation as before, we define a connection on a bundle $\Omega\times W\to\Omega$ as a map $D\colon C^1(\Omega,W)\to C_1(\Omega,W)$ that can be written with an $A\in C_1 (\Omega,\Hom (W,W))$ as
\begin{equation}
D\varphi (v,\xi)=d\varphi (v,\xi)+A (v,\xi)\varphi(v),\qquad \varphi\in C^1 (\Omega,W).
\end{equation}
An alternative notation for $D\varphi (\cdot,\xi)$ is $D_\xi\varphi$.
(4.1) can be abbreviated to $D=d+A$.
The connection is $C^k$ if its connection form $A$ is;
then $D$ maps $C^{k+1} (\Omega,W)$ to $C_1^k (\Omega,W)$.

The curvature of the connection $D$, assumed to be $C^1$, is the $\Hom (W,W)$ valued 2--form
\begin{equation}
R=dA+A\wedge A,
\end{equation}
that can also be defined by
\begin{equation}
R(\cdot,\xi,\eta)\varphi=D_\xi D_\eta\varphi-D_\eta D_\xi\varphi,\qquad \xi,\eta\in V,\quad\varphi\in C^2 (\Omega,W).
\end{equation}

We record two transformation formulae for $A$ and $R$.
Denote by $\text{GL}(W)$ the group of invertible elements of $\Hom(W,W)$.

\begin{lem}Let $D$ be a connection on a bundle $\Omega\times W\to\Omega$, of class $C^1$, with connection form $A$ and curvature $R$.

(a)\ Suppose that $\gamma\colon\Omega\to GL(W)$ and its inverse $\gamma^{-1}\colon\Omega\to GL(W)$, viewed as maps $\Omega\to\Hom (W,W)$, are $C^2$.
Define a connection $D^\gamma$ on $\Omega\times W\to\Omega$ by
\begin{equation}
D_\xi^\gamma\varphi=\gamma^{-1} D_\xi (\gamma\varphi),\qquad \xi\in V,\quad\varphi\in C^1 (\Omega,W).
\end{equation}
Then the connection form $A^\gamma$ and the curvature $R^\gamma$ of $D^\gamma$ satisfy
$$
A^\gamma=\gamma^{-1} d\gamma+\gamma^{-1}A\gamma,\qquad R^\gamma=\gamma^{-1} R\gamma.
$$

(b)\ Suppose $V_0$ is a locally convex space, $\Omega_0\subset V_0$ is open, and $F\colon\Omega_0\to\Omega$ is $C^2$.
Define a connection $D^F$ on $\Omega_0\times W\to\Omega_0$ by its connection form $A^F=F^*A$.
The curvature $R^F$ of $A^F$ satisfies
\begin{equation}
R^F=F^*R.
\end{equation} 
\end{lem}

Because of the chain rule, $A^F=F^*A$ implies for $\varphi\in C^1 (\Omega,W)$
\begin{equation}
D^F (F^*\varphi) (v,\xi)=D\varphi (F(v),F_*|_v\xi),
\qquad v\in\Omega_0,\quad\xi\in V_0.
\end{equation}
If $F$ is a diffeomorphism, then conversely, (4.6) implies $A^F=F^*A$.

\begin{proof}(a) follows from (4.1) and (4.3), while (b) follows from (4.2) in conjunction with Lemma 3.1.
\end{proof}

\section{Riemannian metrics and their curvature}

A Riemannian metric on $\Omega$ is the specification of a positive definite symmetric bilinear form on each tangent space $T_v\Omega$; in other words, a map $g\colon\Omega\times V\times V\to\bR$ with $g_v=g(v,\cdot,\cdot)$ positive definite symmetric bilinear form for all $v\in\Omega$.
The metric is $C^k$ if the map $g$ is.
Clearly, $g$ is uniquely determined if we know the length $|\xi|_v=g(v,\xi,\xi)^{1/2}$ of each tangent vector $\xi\in T_v\Omega\approx V$.
Given a $C^1$ Riemannian metric $g$ on $\Omega$, a Levi--Civita connection $D$ on $T\Omega\approx \Omega\times V\to\Omega$ is a connection with connection form $A$ that is compatible with the metric and has no torsion:\  for $\xi,\eta\in V$ and $\varphi,\psi\in C^1 (\Omega,V)$
\begin{gather}
\xi g(\cdot,\varphi, \psi)=g(\cdot,D_\xi\varphi, \psi)+g(\cdot,\varphi, D_\xi\psi),\\
D_\xi\eta=D_\eta\xi\qquad\text{or equivalently}, \quad A(v,\xi)\eta=A(v,\eta)\xi\text{ for }v\in\Omega.
\end{gather}
In the first formula in (5.2) $\xi,\eta$ represent the corresponding constant vector fields $\Omega\to V$. An alternative of (5.1) is a formula for $dg:T(\Omega\times V\times V)\approx \Omega\times V^{\oplus 5}\to\bR$:
\begin{equation}
dg(v,\eta,\zeta;\xi,\lambda,\mu)=g\big(v,A(v,\xi)\eta,\zeta\big)+g\big(v,\eta,A(v,\xi)\zeta\big)+g(v,\lambda,\zeta)+g(v,\eta,\mu).
\end{equation}
That (5.1) and (5.3) are equivalent follows from the chain rule. (5.3) shows that if the connection is $C^k$, then $g$ is $C^{k+1}$.

In general, a Levi--Civita connection may not exist, but if it does, it is unique.
Indeed, applying (5.1) with constant vector fields $\varphi\equiv\eta$, $\psi\equiv\zeta$ (or using (5.3)), the connection form $A$ will satisfy
$$
\xi g(\cdot,\eta,\zeta)=g(\cdot,A(\cdot,\xi)\eta,\zeta)+g(\cdot,\eta,A(\cdot,\xi)\zeta).
$$ 
Upon permuting $\xi,\eta,\zeta$ cyclically we obtain three equations for six unknowns, namely $g(v,A(v,\xi)\eta,\zeta)$ and permutations.
But (5.2) reduces the number of unknowns to three, which are then  determined by the three equations we have.

Whenever a metric $g$ admits a Levi--Civita connection $D$ of class $C^1$, by the curvature of $g$ we mean the curvature $R$ of $D$, given in (4.2) or (4.3). The Riemann tensor $\cR(v,\xi,\eta,\zeta,\theta)=g\big(v,R(v,\xi,\eta)\zeta,\theta)\big)$ has the usual algebraic symmetry properties, which imply that
$$
K(v,\xi,\eta)=\cR(v,\xi,\eta,\eta,\xi),
$$
if $\xi,\eta$ are orthonormal, depends only on the plane $P\subset T_v\Omega$ spanned by $\xi,\eta$. This is the sectional curvature along $P$.

Consider now another locally convex space $V'$, an open $\Omega'\subset V'$ endowed with a Riemannian metric $g'$, and as before, when $\eta\in T_w \Omega'\approx V'$ set $|\eta|'_w=g'_w (\eta,\eta)^{1/2}$.
A $C^k$ isometry $F\colon\Omega\to\Omega'$ is a $C^k$ diffeomorphism on an open subset of $\Omega'$ such that
$$
|F_* \xi|'_{F(v)}= |\xi|_v, \qquad v\in\Omega,\quad\xi\in T_v\Omega.
$$

\begin{lem}Suppose that the metrics $g,g'$ on $\Omega,\Omega'$ admit Levi--Civita connections $D,D'$ of class $C^1$.
Let $F\colon\Omega\to\Omega'$ be a $C^2$ isometry between $g$ and $g'$, and view $F_*$ as a $\Hom(V,V')$ valued function on $\Omega$.
Then the connection and curvature forms $A,A'$ and $R,R'$ of $D,D'$ satisfy for $v\in\Omega$, $\xi,\eta,\zeta\in T_v\Omega$
\begin{gather}
F_*|_v^{-1}A' \big(F(v),F_*\xi\big) F_*|_v+ F_*|_v^{-1}(\xi F_*)|_v=A(v,\xi)\\
 R' \big(F(v), F_*\xi, F_*\eta\big) F_*\zeta=F_* \big(R(v,\xi,\eta)\zeta\big).
\end{gather}
\end{lem}

\begin{proof}We can assume that $V=V'$ and $\Omega'=F(\Omega)$.
Taking (5.4) as the definition of a $\Hom (V,V)$ valued 1--form $A$, we first show that the connection $\nabla=d+A$ is then Levi--Civita.
That $\nabla$ has no torsion is obvious, since $(\xi F_*)\eta=\xi\eta F$ is symmetric in $\xi,\eta$.
To see that $d+A$ is compatible with $g$, for any $\varphi\in C^1 (\Omega,V)$ define $\varphi'\in C^1 (\Omega',V')$ by
$$
F_*|_v\varphi(v)=\varphi' (F(v)),\qquad v\in\Omega.
$$
Since the passage from $D', A'$ to $\nabla,A$ is a combination of the transformations in Lemma 4.1,
$$
(D'_{F_*\xi}\varphi') (F(v))=(\nabla_\xi\varphi)(v)
$$
by (4.4), (4.6).
Hence 
\begin{eqnarray*}
\xi g(\cdot,\varphi,\psi)=(F_*\xi) g'(F,\varphi',\psi')&=&g' (F,D'_{F_*\xi}\varphi',\psi')+g' (F,\varphi', D'_{F_*\xi}\psi')\\
&=& g(\cdot, \nabla_\xi \varphi,\psi) +g(\cdot,\varphi, \nabla_\xi \psi),
\end{eqnarray*}
and $\nabla=d+A$ is indeed Levi--Civita; in other words, $A$ given in (5.4) is the connection form of $D$.

At this point (5.5) follows by putting together the two transformation formulae of the curvature in Lemma 4.1.
\end{proof}

\begin{cor}Suppose that metrics on $\Omega,\Omega'$ admit Levi--Civita connections of class $C^k$, $k=0,1,\ldots$. Then any $C^2$ isometry $\Omega\to\Omega'$ is automatically $C^{k+2}$.
\end{cor}
\begin{proof} We prove by induction on $k$, the base case $k=0$ being vacuous. Suppose the statement is true for $k-1\ge 0$ instead of $k$. Then in (5.4) the first term on the left and the term on the right, as functions of $v,\xi$, are in $C^k(\Omega\times V,\Hom(V,V))$. It follows that so is $\xi F_*|_v$, which means $F$ is $C^{k+2}$.
\end{proof}

\section{Geodesics}

Consider a connection $D$ on a bundle $\Omega\times W\to \Omega$, with connection form $A$.
Given an interval $I\subset\bR$ and a $C^1$ curve $x\colon I\to\Omega$, by a parallel (or horizontal) lift of $x$ we mean a $C^1$ map $\xi\colon I\to W$ such that
\begin{equation}
{d\xi\over dt}+ A\big(x,{dx\over dt}\big)\xi=0.
\end{equation}
If there are other connections as well in play, we will specify ``$D$--parallel lift''.

When $W$ is a Banach space, the linear differential equation (6.1) can be solved with any initial condition $\xi (t_0)=\xi_0\in W$ $(t_0\in I)$, and uniquely at that; but for more general $W$ there is no guarantee that a solution exists or that it is unique, even if $A$ and $x$ are $C^\infty$.
Nevertheless, if it happens that parallel lifts $\xi$ of $x$ exist and are unique for all initial conditions $\xi(t_0)=\xi_0\in W$, we can define parallel transport from $x(t_0)$ to $x(t_1)$ along $x$: this is the map $W\ni\xi(t_0)\mapsto\xi (t_1)\in W$.

\begin{lem}Suppose  $x_n\in C^1( I,\Omega)$ converge in the $C^1$ topology to $x\colon I\to\Omega$, and their parallel lifts  $\xi_n\in C^1(I,W)$ converge uniformly to $\xi\colon I\to W$. Then $\xi$ is $C^1$, and a parallel lift of $x$.
\end{lem}
\begin{proof}For $\alpha,\beta\in I$ 
$$
\xi\big|_\alpha^\beta=\lim_n\int_\alpha^\beta\frac{d\xi_n}{dt}dt=-\lim_n\int_\alpha^\beta A\big(x_n,\frac{dx_n}{dt}\big)\xi_n\,dt=-\int_\alpha^\beta A\big(x,\frac{dx}{dt}\big)\xi \,dt.
$$
This shows that $\xi$ is continuously differentiable and satisfies (6.1).
\end{proof}

Consider next a Riemannian metric $g$ on $\Omega\subset V$ admitting a Levi--Civita connection $D=d+A$.
A $C^2$ curve $x\colon I\to\Omega$ is a geodesic if $dx/dt$ is a parallel lift of $x$, i.e.
\begin{equation}
{d^2 x\over dt^2} +A\big(x, {dx\over dt}\big) {dx\over dt}=0.
\end{equation}
Again, neither existence nor uniqueness is guaranteed for geodesics, but regularity is:

\begin{lem}If the Levi--Civita connection is of class $C^k$, $k\geq 1$, then geodesics will be $C^{k+2}$.
\end{lem}

The proof is by straightforward induction.---As in finite dimensions, geodesy can be defined for $C^1$ curves as well, and this is of some importance in [Le]. Suppose $g$ is a continuous metric on $\Omega$ and $I=[\alpha,\beta]\subset\bR$. The energy of a curve $x\in C^1(I,\Omega)$ is
$$
E(x)=\frac12\int_\alpha^\beta g\Big(x(t),\frac{dx(t)}{dt},\frac{dx(t)}{dt}\Big)dt.
$$
Two applications of Example 2.2 show that if $g$ is $C^1$ then $E$ is a $C^1$ function on $C^1(I,\Omega)$. Given $a,b\in\Omega$, let
$$
C_{ab}^1(I,\Omega)=\{x\in C^1(I,\Omega)\colon x(\alpha)=a,\, x(\beta)=b\}.
$$

The critical points of $E|C_{ab}^1(I,\Omega)$ we call energy critical curves.

\begin{thm}Let $g$ be a  Riemannian metric on $\Omega$ that admits a Levi--Civita connection.

(a)\, A curve $x\in C^2(I,\Omega)$ is energy critical if and only if it is geodesic.

(b)\, Suppose that there is a family $L$ of continuous linear forms on $V$ such that $\bigcap\{\Ker l\colon l\in L\}=(0)$, and for every $l\in L$ and $v\in\Omega$ there is an $\eta\in V$ such that $l=g(v,\eta,\cdot)$. Then any energy critical $x\in C^1(I,\Omega)$ is $C^2$, and so it is a geodesic.
\end{thm}

Under rather stronger assumptions in part (b) a quick proof could be given along the lines of du Bois Reymond's theorem, see \cite[pp. 172--173]{CH}; but the proof below is not too difficult, either.
\begin{proof}For brevity we denote $d/dt$ derivative by a dot. Using (5.3) we compute the differential of $E$:
\begin{equation}\begin{split}
dE(x,y)&=\frac12\frac d{ds}\Big|_{s=0}\int_\alpha^\beta g(x+sy,\dot x+s\dot y,\dot x+s\dot y)\\
&=\int_\alpha^\beta\bigl\lbrace g\big(x,A(x,y)\dot x,\dot x\big)+g(x,\dot x,\dot y)\bigr\rbrace=\int_\alpha^\beta g\big(x,\dot y+A(x,\dot x)y,\dot x\big),
\end{split}\end{equation}
the last equality because the connection has no torsion. Supposing $x\in C^2(I,\Omega)$, we follow Lagrange and integrate by parts to obtain
\begin{equation}
dE(x,y)=g(x,y,\dot x)\big|_\alpha^\beta -\int_\alpha^\beta g\big(x,y,\ddot x+A(x,\dot x)\dot x\big).
\end{equation}
Therefore $x$ is energy critical if and only if the integral in (6.4) vanishes for all $y\in C_{00}^1(I,V)$, and this latter is equivalent to (6.2).

Suppose next that $x$ is energy critical, but only $C^1$. We cannot integrate by parts, still, the last integral in (6.3) must vanish for all $y\in C_{00}^1(I,V)$. With $\alpha\le\sigma<\tau\le\beta$ and $0<\var<(\tau-\sigma)/2$ define $\chi\colon I\to\bR$ by
$$
\chi(t)=\begin{cases} 0 \text{ if}\quad t\le\sigma\text{ or }t\ge\tau\\ 1 \text{ if}\quad \sigma+\var\le t\le\tau-\var,\end{cases}\qquad\chi\text{ is linear on }[\sigma,\sigma+\var]\text{ and on }[\tau-\var,\tau].
$$
With an arbitrary $\eta\in V$ let $y=\chi\eta$. While $y\notin C_{00}^1(I,V)$,
\begin{equation}
\int_\alpha^\beta g\big(x,\dot y+A(x,\dot x)y, \dot x\big)=0
\end{equation}
still holds. Indeed, if we choose uniformly Lipschitz $\chi_k\in C_{00}^1(I,V)$ such that $\chi_k\to\chi$ almost everywhere, then (6.5) holds with $y_k=\chi_k\eta$ in place of $y$. Hence letting $k\to\infty$ (6.5) follows. In turn, letting $\var\to 0$ in (6.5), 
\begin{equation}
g(x,\eta,\dot x)\big|_\sigma^\tau +\int_\sigma^\tau g\big(x,A(x,\dot x)\eta,\dot x\big)=0
\end{equation}
follows. Of course, this holds also when $\sigma>\tau$. Writing
$$
g(x,\eta,\dot x)\big|_\sigma^\tau=g(x(\sigma),\eta,\dot x)\big|_\sigma^\tau+g(x,\eta,\dot x(\sigma))\big|_\sigma^\tau+g(x,\eta,\dot x\big|_\sigma^\tau)\big|_\sigma^\tau,
$$
as $\tau\to\sigma$, the last term on the right is $o(\sigma-\tau)$ and the penultimate term is 
$$
\sim(\tau-\sigma)\frac d{dt}\Big|_{t=\sigma}g\big( x(t),\eta,\dot x(\sigma)\big)=(\tau-\sigma)\big\lbrace g\big(x,A(x,\dot x)\eta,\dot x\big)+g\big(x,\eta,A(x,\dot x)\dot x\big)\big\rbrace\big|_\sigma.
$$
Dividing in (6.6) by $\tau-\sigma$ and letting $\tau\to\sigma$ gives therefore
$$
\lim_{\tau\to\sigma}g\Big(x(\sigma),\eta,\frac{\dot x(\tau)-\dot x(\sigma)}{\tau-\sigma}\Big)=-g\big(x,\eta,A(x,\dot x)\dot x\big)\big|_\sigma.
$$
In particular, for any $l\in L$ and $t\in [\alpha,\beta]$
\begin{align*}
\frac d{dt}l(\dot x(t))&=\lim_{\tau\to t} l\Big(\frac{\dot x(\tau)-\dot x(t)}{\tau-t}\Big)=-l\big(A(x(t),\dot x(t))\dot x(t)\big),\qquad\text{and so}\\
l\big(\dot x\big|_\sigma^\tau\big)&=\int_\sigma^\tau\frac d{dt}l(\dot x(t))\, dt=-l\Big(\int_\sigma^\tau A\big(x(t),\dot x(t)\big)\dot x(t)\, dt\Big).
\end{align*}
Hence $\dot x(\tau)-\dot x(\sigma)=-\int_\sigma^\tau A(x,\dot x)\dot x$, and $\dot x$ is indeed continuously differentiable.
\end{proof}
\begin{cor}Consider a second locally convex space, a continuous Riemannian metric $g'$ on an open $\Omega'\subset V'$, and a $C^1$ isometry $F\colon\Omega'\to\Omega$. Suppose the metric $g$ on $\Omega$ is as in Theorem 6.3b. If $y\in C^1(I,\Omega')$ minimizes energy in $C^1_{y(\alpha)y(\beta)}(I,\Omega')$, then $x=F\circ y$ is a $C^2$ geodesic.
\end{cor}
\begin{proof}We can assume $F(\Omega')=\Omega$; then $x\in C^1(I,\Omega)$ minimizes energy in $C_{x(\alpha)x(\beta)}^1(I,\Omega)$, hence it is a $C^2$ geodesic by Theorem 6.3b.
\end{proof}
Parallel lifts and geodesics transform under isometries as expected, assuming the isometry is $C^2$:

\begin{lem}Let $g,g'$ be Riemannian metrics on $\Omega\subset V$, $\Omega'\subset V'$  with Levi--Civita connections $D,D'$ of class $C^1$.
Suppose $F\colon\Omega'\to\Omega$ is a $C^2$ isometry.
If $y\colon I\to\Omega'$ is a $C^1$ curve and $\eta\colon I\to V'$ its $D'$--parallel lift, then $F_*|_y \eta=\xi\colon I\to V$ is a $D$--parallel lift of $F\circ y$.
In particular, if $y$ is a geodesic for $g'$, then $F\circ y$ is a geodesic for $g$.
\end{lem}

\begin{proof}Since by Lemma 5.1 we know how the connection forms $A,A'$ of $D,D'$ are related, the first statement follows in a straightforward way from (5.4).
The second follows from the first, applied with $\eta=dy/dt$.
\end{proof}

\section{Geodesics, curvature, and parallel transport}

Lemma 5.1 describes how curvature transforms under $C^2$ isometries, and Lemma 6.5 how parallel lifts transform. Sometimes it is possible to prove the same  for isometries that are merely $C^1$, namely in spaces of K\"ahler potentials that we study in [Le]. 
The reason is that curvature and parallel lifts can be explained in terms of geodesics and Jacobi fields, as we presently show, and geodesics tend to be preserved by $C^1$ isometries, cf. Corollary 6.4.

Let $g$ be a Riemannian metric on $\Omega\subset V$ and $D=d+A$ its Levi--Civita connection. If $I\subset\bR$ is an interval, $x\in C^1(I,\Omega)$,   $\varphi\in C^1(I,V)$, and $t\in I$, we write
$$
D_{\dot x(t)}\varphi(t)=\dot\varphi(t)+A\big(x(t),\dot x(t)\big)\varphi(t),
$$
where dot still means $d/dt$.  Suppose  $x$ is a geodesic. Provided the Levi--Civita connection is $C^1$, we define a Jacobi field along $x$ as a solution $\varphi\in C^2(I,V)$ of the ODE
$$
D^2_{\dot x}\varphi=R(x,\dot x,\varphi)\dot x.
$$
If for $\theta$ in a   neighborhood of $0\in\bR$ we are given geodesics $x_\theta:I\to\Omega$ so that $x_0=x$, and the map $(t,\theta)\mapsto x_\theta(t)$ is 
$C^3$, the variation $\varphi=dx_\theta/d\theta|_{\theta=0}$ is a Jacobi field along  $x$, as one checks by taking $d/d\theta$ of the geodesic equation $D_{\dot x_{\theta}}\dot x_{\theta}=0$. 

Consider a disc $\Delta\subset\bR^2$ centered at the origin and a map $e\in C^5(\Delta,\Omega)$, whose  restriction to any radius of $\Delta$ is a unit speed geodesic. This implies that  $e_*$ is an isometry between the Euclidean metric on $T_0\Delta$ and $g$ on $e_*T_0\Delta$. 

\begin{lem}Assuming that the Levi--Civita connection is $C^3$, the length of the curve $[0,2\pi]\ni\theta\mapsto e(r\cos\theta, r\sin\theta)$ is
$$
2\pi r\big(1-Kr^2/6+o(r^2)\big)\qquad\text{as }r\to 0,
$$
where $K$ denotes sectional curvature  of $g$ along the plane $e_*T_0\Delta$.
\end{lem}
\begin{proof}We need to compute the length of the vectors $\partial e(r\cos\theta,r\sin\theta)/\partial\theta$. Let 
$$ 
x_\theta(t)= e(t\cos\theta,t\sin\theta)\qquad\text{and}\qquad \varphi_\theta=\partial x_\theta/\partial\theta, 
$$
so that $x_\theta$ is a geodesic and $\varphi_\theta$ is a Jacobi field along it,
\begin{equation}
D^2_{\dot x_\theta}\varphi_\theta=R(x_\theta,\dot x_\theta,\varphi_\theta)\dot x_\theta.
\end{equation}
In the calculation to follow we will omit the subscript $\theta$. By the chain rule, if we apply $D_{\dot x}$ to (7.1), we obtain at points where $\varphi$ vanishes
\begin{equation}
D^3_{\dot x}\varphi=R(x,\dot x,D_{\dot x}\varphi)\dot x.
\end{equation}
Since $\varphi(0)=0$, this allows us to compute $D^j_{\dot x}\varphi(0)$ for $j=0,1,2,3$. The isometry of $e_*|_0$ implies that $\dot x(0)=\xi$ and $D_{\dot x}\varphi(0)=\dot\varphi(0)=\eta$ form an orthonormal basis of $e_*T_0\Delta$. Thus, by (7.1), (7.2)
\begin{equation}
\varphi(0) = 0,\quad D_{\dot x}\varphi(0)=\eta,\quad D^2_{\dot x}\varphi(0)=0,\quad D^3_{\dot x}\varphi(0)=R(v,\xi,\eta)\xi.
\end{equation}
Setting $h(t)=g\big(x(t),\varphi(t),\varphi(t)\big)$, by Taylor's formula
$$
h(r)=\sum_{j=0}^3 h^{(j)}(0)\frac{r^j}{j!}+\int_0^r h^{(4)}(t)\frac{(r-t)^3}{3!}\, dt=\sum_{j=0}^4 h^{(j)}(0)\frac{r^j}{j!}+o(r^4),\qquad r\to 0,
$$
with $o(r^4)$ uniform for $0\le\theta\le 2\pi$. In turn,
$$
h^{(j)}=\sum_{i=0}^j \binom ji g\\(x, D^i_{\dot x}\varphi,D^{j-i}_{\dot x}\varphi).
$$
The values in (7.3) therefore give $h(0)=h'(0)=h'''(0)=0$, $h''(0)=2$, and  $h^{(4)}(0)=8g\big(v,R(v,\xi,\eta)\xi,\eta\big)=-8K$,  so that $h(r)=r^2\big(1-Kr^2/3+o(r^2)\big)$. Hence the length in question is
$$
\int_0^{2\pi}g\big(x_\theta(r),\varphi_\theta(r),\varphi_\theta(r)\big)^{1/2}\, d\theta=2\pi r\big(1-Kr^2/6+o(r^2)\big).
$$
\end{proof}

\begin{lem} Let $x\in C^2(I,\Omega)$ be a geodesic, $\xi\in C^1(I,V)$, and $\tau\in I$. 
Suppose  that there is a collection $H$ of Jacobi fields $\eta$ along $x|J$ with $J=J_{\eta}\subset I$ a neighborhood of $\tau$, such that each $\eta$ vanishes at $\tau$ and $\{\dot\eta(\tau)\colon\eta\in H\}$ is dense in $V$. Then $g(x,\xi,\eta)$ for any $\eta\in H$  is twice differentiable at $\tau$, and furthermore  $D_{\dot x}\xi(\tau)=0$ if and only if
\begin{equation}
\frac{d^2}{dt^2}g\big(x(t),\xi(t),\eta(t)\big)|_{t=\tau}=0\qquad\text{for all }\eta\in H.
\end{equation}
\end{lem}
Note that $\dot\eta(\tau)=D_{\dot x}\eta(\tau)$ if $\eta$ vanishes at $\tau$.
\begin{proof}For any $\eta\in H$
$$
\frac d{dt} g(x,\xi,\eta)=g(x,D_{\dot x}\xi,\eta)+g(x,\xi,D_{\dot x}\eta).
$$
The derivative of the second term is $g(x,D_{\dot x}\xi,D_{\dot x}\eta)+g(x,\xi,D^2_{\dot x}\eta)=g(x,D_{\dot x}\xi,D_{\dot x}\eta)$, and of the first term, at $\tau$,  is
$$
\lim_{t\to\tau} g\big(x(t),D_{\dot x}\xi(t),\eta(t)/(t-\tau)\big)=g\big(x(\tau),D_{\dot x}\xi(\tau),\dot\eta(\tau)\big).
$$
Hence (7.4) reduces to
$$
g\big(x(\tau),D_{\dot x}\xi(\tau),\dot\eta(\tau)\big)=0\qquad\text{for }\eta\in H,
$$
and by density in fact to $g\big(x(\tau),D_{\dot x}\xi(\tau), v\big)=0$ for all $v\in V$. But this latter is equivalent to $D_{\dot x}\xi(\tau)=0$.
\end{proof}

\end{document}